\documentclass[11pt]{amsproc}
\usepackage{amssymb,amsmath,amsthm,tikz}

\usepackage[normalem]{ulem}
\newtheorem{teorema}{Theorem}[section]

\newtheorem{ejemplo}[teorema]{Example}

\newtheorem{teo}[teorema]{Theorem}
\newtheorem{prop}[teorema]{Proposition}
\newtheorem{remark}[teorema]{Remark}

\usepackage{lineno}

\usepackage[pagebackref]{hyperref}

\begin{document}
	
\title{On the Zariski invariant of plane branches}
\author{{\sc Marcelo Escudeiro Hernandes and} \\ {\sc Mauro Fernando Hern\'andez Iglesias}}\thanks{The first-named author was partially supported by CNPq-Brazil and the second-named author were partially supported by  the Direcci\'on de Fomento de la Investigaci\'on at the PUCP through grant DFI-2023-PI0983.}
	
\date{}

\dedicatory{Dedicated to the memory of Professor Arkadiusz P\l oski }
 
\begin{abstract}
We show how to obtain the Zariski invariant of a plane branch employing the contact order or the intersection multiplicity with elements in a particular family of curves
and we present some consequences of this result.
\end{abstract}
	
  \maketitle

 \noindent {\it Keywords: Plane branches, Zariski invariant, intersection multiplicity}. \\
		\noindent {\it 2020 AMS Classification:} 14H20 (primary); 14H50, 14B05 (secondary).\\

\markboth{M. E. Hernandes and M. F. Hern\'andez Iglesias}{On the Zariski invariant of plane branches}

\section{Introduction}

Let $C_f:\{f=0\}$ be an irreducible singular plane curve (simply a plane branch) defined by an irreducible convergent power series $f\in \mathbb C\{x,y\}$. We denote by $\text{mult}(h)$ the multiplicity of $h\in\mathbb{C}\{x,y\}\setminus\{0\}$, that is the smallest  $s\in\mathbb{N}$ such that $h\in\mathcal{M}^s\setminus\mathcal{M}^{s+1}$, where $\mathcal{M}$ stands for the maximal ideal of $\mathbb{C}\{x,y\}$. Up to change of coordinates, we may assume that $\{x=0\}$ is transversal to $C_f$ and $f\in\mathbb{C}\{x\}[y]$ is a Weierstrass polynomial that is $f=y^n+\sum_{i=1}^{n}c_i(x)y^{n-i}$ where $c_i(x)\in\mathbb{C}\{x\}$ with $\text{mult}(c_i(x))> i$ and $n=\text{mult}(f)=\deg_y(f)$ is the multiplicity of $f$. We denote $\text{mult}(C_f):=\text{mult}(f)$. 

By the Newton-Puiseux theorem  $f$ admits a root given by $$\alpha(x):=\sum_{i> n}a_ix^{\frac{i}{n}}\in\mathbb{C}\{x^{\frac{1}{n}}\}.$$ In addition, the zero set of $f$ is
$\left\{\alpha_j(x):=\sum_{i> n}a_i\epsilon_j^ix^{\frac{i}{n}}:\ \epsilon_j\in \mathbb{U}_n\right\}$,
where $\mathbb{U}_n$ is the multiplicative group of the complex $n$th roots of the unity. In this way, we get
$$
f(x,y)=\prod_{j=1}^{n}\left ( y-\alpha_j(x)\right ).$$

Setting $t=x^{\frac{1}{n}}$ we have that $\alpha(t^n)=\sum_{i> n}a_it^{i}\in\mathbb{C}\{t\}$. The pair
\begin{equation}\label{param}
 (t^n,\sum_{i> n}a_it^i )  
\end{equation}
is called a \emph{Puiseux parametrization} of $C_f$. Notice that $f(t^n,\sum_{i> n}a_it^i)=0$. In addition, by the Newton-Puiseux theorem, we have that (\ref{param}) is a primitive parametrization,
that is $n$ and the elements in the set $\{ i:\ a_i\neq 0\}$ do not admit a non-trivial common divisor. 

Given a Puiseux parametrization as (\ref{param}) we define two sequences of integers:
\begin{equation}\label{sequences}
\begin{array}{ll}
 \beta_0:=n  & e_0:=n  \\
\beta_j:=\min\{i:\ a_i\neq 0\ \mbox{and}\ i\not\in e_{j-1}\mathbb{N}\}  & e_j:=\gcd(e_{j-1},\beta_j) 
\end{array}    
\end{equation}
for $j>0$.

In what follows we denote $m:=\beta_1$.
Since the parametrization is primitive there exists an integer $g\geq 1$ such that $e_g=1$ and the sequences (\ref{sequences}) are finite.

The sequence $(\beta_i)_{i=0}^{g}$ is called \emph{characteristic sequence} of $C_f$ and it determines the topological type of the curve $C_f$ (see \cite[pag 465]{zariski-equising}). The set of all irreducible plane curves with the same characteristic sequence $(\beta_i)_{i=0}^{g}$ or equivalently, that share the same topological type is denoted by $K(n,m,\beta_2,\ldots ,\beta_g)$.

Given $C_f\in K(n,m,\beta_2,\ldots, \beta_g)$ we consider the set 
	\[
	\Gamma_f:=\{\textup{I}(f,h)\;:\; f \,\hbox{\rm does not divide }h\ \mbox{in}\ \mathbb{C}\{x,y\}\},
	\]
where $\textup{I} (f,h):=\dim_{\mathbb{C}}\frac{\mathbb{C}\{x,y\}}{\langle f,h\rangle}$ is the \emph{intersection multiplicity} of $C_f$ and $C_h$ at the origin. We also denote $\textup{I}(f,h)$ by $\textup{I}(C_f,C_h)$.
It follows by the properties of codimension of ideals that $\Gamma_f$ is an additive semigroup of $\mathbb{N}$ called the \emph{values semigroup} of $C_f$. Moreover, $\Gamma_f$ admits a conductor  $$\mu_f=\min\{\gamma\in\Gamma_f:\ \gamma-1\not\in\Gamma_f\ \mbox{and}\ \gamma+k\in\Gamma_f\ \mbox{for any}\ k\in\mathbb{N}\}$$
and it coincides with the Milnor number of $C_f$, that is $\mu_f=\dim_{\mathbb{C}}\frac{\mathbb{C}\{x,y\}}{\langle f_x,f_y\rangle}$.
 
If $C_f\in K(n,m,\beta_2,\ldots, \beta_g)$ then the semigroup $\Gamma_f$ is finitely generated by $g+1$ natural numbers $v_0<v_1<\ldots <v_g$ and  there is a relationship between the sequences $(\beta_i)_{i=0}^{g}$ and $(v_i)_{i=0}^g$ as follows (see \cite[Theorem 3.9]{Zariski-libro} for instance):
\begin{equation}\label{exp-gen}	\begin{array}{l}
v_0=\beta_0=n,\ \  v_1=\beta_1=m,\\
v_j=n_{j-1}v_{j-1}+\beta_j-\beta_{j-1}\ \mbox{for}\ 2\leq j\leq g, \mbox{where}\ \ n_{j-1}:=\frac{e_{j-2}}{e_{j-1}}.
\end{array}
\end{equation}

In what follows we write
$\Gamma_f=\langle v_0, v_1,\ldots ,v_g\rangle:=\mathbb N v_0+\mathbb Nv_1+\cdots +\mathbb N v_g$.

In addition, according to \cite[Proposition 9.15]{rosales}, the conductor of $\Gamma_f$ can be expressed by $\mu_f=\sum_{i=1}^{g}(n_i-1)v_i-(v_0-1)$.

Let $C_f\in K(n,m,\beta_2,\ldots ,\beta_g)$ be a plane branch with Puiseux parametrization $(t^n,\sum_{i>n}a_it^i)$.
Zariski (\cite[pages 785-786]{zariski-torsion}) proved that if $a_j\neq 0$, $n<j\neq m$ and $j+n\in\Gamma_f$ then there exists a change of coordinates such that $C_f$ is analytically equivalent to a plane branch with Puiseux parametrization
$$\left ( t^n, \sum_{n<i<j}a_it^i+\sum_{i>j}a'_it^i\right ).$$
Moreover, he showed that $C_f$ is analytically equivalent to a plane branch with Puiseux parametrization $(t^n,t^m)$ or

\begin{equation}\label{zariski}
\left (t^n,t^m+b_{\lambda_f}t^{\lambda_f}+\sum_{i> \lambda_f}b_it^i\right )\ \  \mbox{with}\  b_{\lambda_f}\neq 0\ \mbox{and}\ \lambda_f+n\not\in\Gamma_f.
\end{equation} The integer $\lambda_f$
is an analytical invariant (see page 785 in \cite{zariski-torsion}) called the \emph{Zariski invariant} of $C_f$. If $C_f$ is analytically equivalent to $(t^n,t^m)$ we put $\lambda_f=\infty$.  

In general, it is not immediate to identify the Zariski invariant directly by any Puiseux parametrization as we illustrate in the following example.

\begin{ejemplo}\label{example} Let us consider $C_f\in K(4,7)$ given by the Puiseux parame\-trization $$\varphi(t):=\left (t^4,t^7+t^{10}+t^{12}+bt^{13}\right ).$$
	
Notice that $10+4,12+4\in\Gamma_f$ and $13+4\not\in\Gamma_f$. But we can not conclude that $\lambda_f=13$ for any $b\neq 0$. 

In fact, taking the change of coordinates 
\[\sigma(x,y)=\left (x+\frac{4}{7}y,y-x^3\right )\]
and parameter \[\rho(t)=t-\frac{1}{7}t^4-\frac{3}{98}t^7-\frac{1}{7}t^9\]
we get 
\[
\psi(t):=\sigma\circ\varphi\circ\rho(t)=\left (t^4+\left ( \frac{4}{7}b-\frac{32}{49}\right )t^{13}+A(t), t^7+\left (b-\frac{17}{14}\right )t^{13}+B(t) \right ),\]
where $A(t), B(t)\in\mathbb{C}\{t\}$ have order greather then $13$.

Now considering the change of parameter $$t_1:=t\cdot \left ( 1+\left ( \frac{4}{7}b-\frac{32}{49}\right )t^{9}+\frac{A(t)}{t^4}\right )^{\frac{1}{4}}$$
we have that $C_f$ is analytically equivalent to the plane branch with parametrization 
\begin{equation}\label{pre}
\left (t_1^4, t_1^7+\left (b-\frac{17}{14}\right )t_1^{13}+S(t_1)\right )
\end{equation}
where $S(t_1)\in\mathbb{C}\{t_1\}$ has order greather then $13$.

Since any integer $z>13$ is such that $z+4\in\Gamma_f$, by a change of coordinates and parameter (see \cite[pag. 784]{zariski-torsion}), any term in (\ref{pre}) with order greater than $13$ can be eliminated and consequently, $C_f$ is analytically equivalent to a plane branch with Puiseux parametrization
\[
\left (t^4, t^7+\left (b-\frac{17}{14}\right )t^{13}\right ).
\]
In this way,  we get $\lambda_f=13$ if and only if $b\neq\frac{17}{14}$, that is $C_f$ is analytically equivalent to the plane branch defined by $y^4-x^7=0$ if and only if $b=\frac{17}{14}$.
\end{ejemplo}

In this paper we characterize the Zariski invariant by means of the contact order and the intersection multiplicity with a particular family of plane branches and we present some consequences of this result.
	
\section{The Zariski invariant, contact and intersection multiplicity of plane branches}	

As before, we consider $C_f\in K(n,m,\beta_2,\ldots, \beta_g)$ with a Puiseux parame\-trization $(t^n,\sum_{i> n}a_it^i)$, where $\Gamma_f=\langle n,m,v_2,\ldots ,v_g\rangle$ is its values semigroup and
$\lambda_f$ is the Zariski invariant of $C_f$ as introduced in (\ref{zariski}).  

\begin{remark}\label{remark-std-expression}
If $\Gamma_f=\langle n,m,v_2,\ldots ,v_g\rangle$ and $n_i=\frac{e_{i-1}}{e_i}$ for $i=1,\ldots ,g$ as (\ref{exp-gen}) then any $z\in\mathbb{Z}$ can be uniquely represented (see \cite[Lemma 9.14]{rosales}) as	\begin{equation}\label{standard}
		z=\sum_{i=0}^{g}s_iv_i\ \ \mbox{with}\ \  0\leq s_i< n_i\ \ \mbox{for}\ \ 1\leq i\leq g\ \ \mbox{and}\ \ s_0\in\mathbb{Z}.
	\end{equation}
In particular, an integer $z=\sum_{i=0}^{g}s_iv_i$ as in \eqref{standard} belongs to $\Gamma_f$ if and only if $s_0\geq 0$.    
\end{remark}
	
Let $C_f\in K(n,m,\beta_2,\ldots, \beta_g)$ be a plane branch with $g\geq 2$ and Puiseux parametrization $(t^n,\sum_{i> n}a_it^i)$. By (\ref{exp-gen}), we get $\beta_2=v_2+\beta_1-n_1v_1=v_2+v_1-m_1v_0$ with $m_1:=\frac{m}{e_1}>2$. In particular, $\beta_2+v_0=v_2+v_1-(m_1-1)v_0$, so $\beta_2+v_0$ is an integer as in \eqref{standard}, with $-(m_1-1)=s_0<0$, hence $\beta_2+v_0\not\in \Gamma_f$. Since, by (\ref{sequences}), we get $a_{\beta_2}\neq 0$ and for $g\geq 2$ we have $$m<\lambda_f\leq\beta_2.$$

Let us recall the notion of contact order between two branches.

Let $C_f$ and $C_h$ be two plane branches defined by Weierstrass polynomials $f, h\in\mathbb{C}\{x\}[y]$ with $n=\text{mult}(f)=\deg_y(f)$ and $n'=\text{mult}(h)=\deg_y(h)$. If $\{\alpha_i(x):\ 1\leq i\leq n\}$ and $\{\delta_j(x):\ 1\leq j\leq n'\}$ denote the zero set of $f$ and $h$ respectively, then the \emph{contact order} of $C_f$ with $C_h$ is defined as
\begin{equation}\label{contact}
\text{cont}(C_f,C_h)=
		\max_{1\leq i \leq n \atop	1\leq j \leq n'}
		\left\{\text{mult}\left(\alpha_{i}(x)-\delta_{j}(x)\right)\right\}.
	\end{equation}

The following proposition relates the contact order and the intersection multiplicity of two branches:

\begin{prop}(\cite[Proposition 2.4]{merle})\label{cont-I}
		Let $C_f\in K(n,m,\beta_2,\ldots,\beta_g)$, $\Gamma_f=\langle n,m,v_2,\ldots,v_g\rangle$ its values semigroup and $C_h$ be any plane branch. The following statements are equivalent:\vspace{0.2cm}
		\begin{enumerate}
			\item[i)] $\text{cont}(C_f,C_h)=\theta$, with $\theta\in \mathbb Q$ and $\frac{\beta_q}{n}\leq \theta <\frac{\beta_{q+1}}{n}$ for some $1\leq q\leq g$
			
			\hspace{6cm} (by convention $\beta_{g+1}=\infty)$.\vspace{0.2cm}
			\item[ii)] $\dfrac{\textup{I} (C_f,C_h)}{\text{mult}(C_h)}=\dfrac{n_qv_q+n\theta-\beta_q}{n_0n_1\cdots n_q}$\hspace{1.8cm} (where  $n_0=1$).
		\end{enumerate}
	\end{prop}

Since the intersection multiplicity and the multiplicity of plane branches are invariants by analytical change of coordinates it follows that the contact order between two branches is also an invariant by analytical equivalence.
 
\begin{remark}\label{triangular}
A direct application of the contact formula (\ref{contact}) shows that for three plane branches $C_1, C_2$ and $C_3$ we have that at least $2$ of the three values
$$\text{cont}(C_1,C_2),\  \text{cont}(C_1,C_3),\ \text{cont}(C_2,C_3)$$ are equal and the third one is not smaller than the other two. In addition, according to P\l oski \cite[Th\'eor\`em 1.2]{Ploski}, we get that at least $2$ of the three values
$$\frac{\textup{I}(C_1,C_2)}{\text{mult}(C_1)\text{mult}(C_2)},\  \frac{\textup{I}(C_1,C_3)}{\text{mult}(C_1)\text{mult}(C_3)},\ \frac{\textup{I}(C_2,C_3)}{\text{mult}(C_2)\text{mult}(C_3)}$$  are equal and the third one is not smaller than the other two. This property is
known in the literature as \emph{triangular inequality}.
\end{remark}

Notice that the integers $\beta_0, \beta_1,\ldots ,\beta_g$ and $v_0, v_1,\ldots ,v_g$ associated to a plane branch $C_f\in K(\beta_0,\beta_1,\ldots ,\beta_g)$ are geometrically characterized by
\begin{equation}\label{charct-cont}
\begin{array}{l}
	\beta_0=\min\{\text{cont}(C_f,C)\ :\ C\ \mbox{is a regular curve}\} \vspace{0.2cm}\\
	\beta_1=\max\{\text{cont}(C_f,C)\ :\ C\ \mbox{is a regular curve}\}
	\vspace{0.2cm}\\
	\beta_i=\max\left \{\text{cont}(C_f,C)\ :\ C\in K(\frac{\beta_0}{e_{i-1}},\ldots ,\frac{\beta_{i-1}}{e_{i-1}})\right \}\ \ \mbox{for}\ 2\leq i\leq g
\end{array}
\end{equation}
and
\begin{equation}\label{vi-inter}
\begin{array}{l}
v_0=\min\{\textup{I}(C_f,C)\ :\ C\ \mbox{is a regular curve}\} \vspace{0.2cm}\\
v_1=\max\{\textup{I}(C_f,C)\ :\  C\ \mbox{is a regular curve}\}
\vspace{0.2cm}\\
v_i=\max\left \{\textup{I}(C_f,C)\ :\ C\in K(\frac{\beta_0}{e_{i-1}},\ldots ,\frac{\beta_{i-1}}{e_{i-1}})\right \}\ \ \mbox{for}\ 2\leq i\leq g.
\end{array}
\end{equation}

In what follows, similar to (\ref{charct-cont}) and (\ref{vi-inter}), we present a geometric interpretation for the Zariski invariant of a plane branch using the contact order or the intersection multiplicity with elements in a family $\mathcal{B}$ of curves in $K(n_1,m_1)$.

\begin{teo}\label{geo}
Let $C_f\in K(n,m,\beta_2, \ldots, \beta_g)$ be a plane branch defined by a Weierstrass polynomial $f\in\mathbb{C}\{x\}[y]$.  Then
		\[\lambda_f =n\cdot\max_{C\in \mathcal B}\{\text{cont}(C_f,C)\}=
		\max_{C\in \mathcal B} \{\textup{I}(C_f,C)\}-(n_1-1)m,\]
		where $\mathcal B\subset K(n_1,m_1)$ is the set of branches which are analytically equivalent to  $y^{n_1}-x^{m_1}=0$ with $n_1=\frac{n}{e_1}$ and $m_1=\frac{m}{e_1}$. 
\end{teo}
\begin{proof}
If the Zariski invariant of $C_f$ is $\lambda_f=\infty$ then, by \cite[pag. 784]{zariski-torsion}, we get $e_1=gcd(n,m)=1$, that is, $n=n_1$, $m=m_1$ and $C_f$ is analytically equivalent to $y^{n_1}-x^{m_1}=0$, consequently $C_f\in \mathcal B$ and the theorem follows since $\textup{I}(C_f,C_f)=\infty=\text{cont}(C_f,C_f)$. 

Let us suppose that $C_f$ has a finite Zariski invariant $\lambda_f$. In this way, there exists an analytic change of coordinates $\Phi$ such that $\Phi(C_f)$ has a Puiseux parametrization as (\ref{zariski}), that is
$$(t^n,t^m+b_{\lambda_f}t^{\lambda_f}+\sum_{i>\lambda_f}b_it^i)\ \ \mbox{with}\ b_{\lambda_f}\neq 0.$$
After Proposition \ref{cont-I} and in order to compute $\max_{C\in \mathcal B }\{\textup{I}(C_f,C)\}$, 
it is enough to determine $\max_{C\in \mathcal B}\{\text{cont}(C_f,C)\}$.

Notice that $C_h\in\mathcal{B}\subset K(n_1,m_1)$ defined by $h=y^{n_1}-x^{m_1}$ whose Puiseux parametrization is $(t^{n_1},t^{m_1})$ is such that $\text{cont}(\Phi(C_f),C_h)=\frac{\lambda_f}{n}$. 

In addition, given $C\in\mathcal{B}\subset K(n_1,m_1)$ if
$\text{cont}(\Phi(C_f),C)>\frac{\lambda_f}{n}=\text{cont}(\Phi(C_f),C_h)$ then by Remark \ref{triangular} we get $\text{cont}(C_h,C)=\frac{\lambda_f}{n}$. By definition of contact order (see \eqref{contact}), $C$ admits a Puiseux parametrization $(t^{n_1},t^{m_1}+c_{k}t^{k}+\sum_{i> k}c_it^i)$ with $\text{cont}(C_h,C)=\frac{\lambda_f}{n}=\frac{k}{n_1}$ for some $k> m_1$ and $c_k\neq 0$, that is, $k=\frac{\lambda_f}{e_1}$.
Since $C\in \mathcal{B}\subset K(n_1,m_1)$ we must have $k+n_1\in\langle n_1,m_1\rangle$ otherwise $k$ would be the Zariski invariant of $C$ that is a contradiction because $C\in\mathcal{B}$ . But in this way, $\lambda_f+n=e_1k+e_1n_1\in \langle n,m\rangle\subseteq \Gamma_f$ that is absurd, because $\lambda_f$ is the Zariski invariant of $C_f$. Hence, $\max_{C\in \mathcal B}\{\text{cont}(\Phi(C_f),C)\}=\frac{\lambda_f}{n}$.

Notice that for any change of coordinates $\Phi$ and for every $C\in\mathcal{B}$ we get $\Phi(C)\in\mathcal{B}$. In particular, $\Phi(\mathcal{B})\subseteq\mathcal{B}$ and $\Phi^{-1}(\mathcal{B})\subseteq\mathcal{B}$, consequently, $\Phi(\mathcal{B})=\mathcal{B}$. Since the contact order is invariant by change of coordinates we get
\[\begin{array}{cclcc}
\frac{\lambda_f}{n} & = & \max_{C\in \mathcal B}\{\text{cont}(\Phi(C_f),C)\} & & \vspace{0.2cm}\\ & = & \max_{\Phi^{-1}(C)\in \Phi^{-1}(\mathcal B)}\{\text{cont}(C_f,\Phi^{-1}(C))\} & = & \max_{C\in \mathcal B}\{\text{cont}(C_f,C)\}.\end{array}\]
This finishes the proof of the first equality of the statement.
		
Since $m<\lambda_f\leq\beta_2$ and $\text{mult}(C)=n_1$ for every $C\in\mathcal{B}$, again by Proposition \ref{cont-I}, we get
\[
\max_{C\in \mathcal B} \{\textup{I}(C_f,C)\}=\left\{\begin{array}{ll}
			(n_1-1)m+\lambda_f, & \hbox{\rm if }\lambda_f<\beta_2\\
			v_2=(n_1-1)m+\beta_2,& \hbox{\rm if }\lambda_f=\beta_2,
		\end{array}
		\right.
		\]
and the theorem follows.
\end{proof} 

\begin{ejemplo}\label{exemplo2}
If $C_f$ is the plane branch with Puiseux parametrization $(t^4,t^7+t^{10}+t^{12})$ then, according to Example \ref{example}, we get $\lambda_f=13$ and the plane branch $C_h$ with parametrization $(t^4,t^7+t^{10}+t^{12}+\frac{17}{14}t^{13})$ is an element in $\mathcal{B}\subset K(4,7)$ such that
$\text{cont}(C_f,C_h)=\frac{13}{4}=\frac{\lambda_f}{n}$. Consequently, the branch $C_h$ satisfies
$$\text{cont}(C_f,C_h)=\max_{C\in \mathcal B}\{\text{cont}(C_f,C)\}\ \ \mbox{and}\ \ \ \textup{I}(C_f,C_h)=
		\max_{C\in \mathcal B} \{\textup{I}(C_f,C)\}.$$
\end{ejemplo}

In \cite[pages 62-63]{Casas} Casas-Alvero studied a similar property of Theorem \ref{geo} using the theory of infinitely near points although no formula is presented in this context.
	
Let $C_f\in K(\beta_0,\beta_1,\ldots,\beta_g)$ and $C_h\in  K(\beta'_0,\beta'_1,\ldots,\beta'_{g'})$ be two plane branches with values semigroup $\Gamma_f=\langle v_0,v_1,\ldots ,v_g\rangle$ and $\Gamma_h=\langle v'_0,v'_1,\ldots ,v'_{g'}\rangle$ respectively.
Using the definition of contact order and a simple computation, it follows that if $\text{cont}(C_f,C_h)=\theta>\frac{\beta_k}{\beta_0}$ then  
\begin{equation}\label{contac-equality}
\frac{e_i}{e'_i}=\frac{\beta_i}{\beta'_i}=\frac{v_i}{v'_i}\ \ \ 
\text{for}\ \ \ 0\leq i\leq k
\end{equation}
where $e'_i=\gcd(\beta'_0,\ldots ,\beta'_i)=\gcd(v'_0,\ldots ,v'_i)$.

As an application of Theorem \ref{geo} we will see that the relationship (\ref{contac-equality}) is also valid for the Zariski invariant.			

\begin{prop}\label{proposition-mais}Let $f,h\in\mathbb{C}\{x\}[y]$ be two irreducible Weierstrass polynomials defining $C_f\in K(n,m,\beta_2,\ldots ,\beta_g)$ and $C_h\in K(n',m',\beta'_2,\ldots ,\beta'_{g'})$ with $\lambda$ and $\lambda'$ their respective Zariski invariants. 
\begin{itemize}
    \item[i)] If $\text{cont}(C_f,C_h) >\frac{\lambda}{n}$ then $\frac{\lambda}{n}=\frac{\lambda'}{n'}$.\vspace{0.1cm}
    \item[ii)] If $\ \textup{I}(C_f,C_h)>n'\cdot\left (\frac{(n_1-1)m+\lambda}{n_1}\right )$ then $\frac{\lambda}{n}=\frac{\lambda'}{n'}$.
\end{itemize}
\end{prop}	
\begin{proof}$ $
\begin{itemize}
 \item[i)] 
Since $\text{cont}(C_f,C_h) >\frac{\lambda}{n}>\frac{m}{n}$ it follows by (\ref{contac-equality}) that $$n_1=\frac{n}{e_1}=\frac{n'}{e'_1}\ \ \mbox{and}\ \ m_1=\frac{m}{e_1}=\frac{m'}{e'_1}.$$

By Theorem $\ref{geo}$ we have that 
$$\hspace{1.5cm}\frac{\lambda}{n}=\text{max}_{C\in \mathcal B}	\{\text{cont}(C_f,C)\}\ \ \mbox{and}\ \  \frac{\lambda'}{n'}=\text{max}_{C\in \mathcal B}	\{\text{cont}(C_h,C)\},$$ where $\mathcal B$ is the set of plane branches which are analytically equivalent to $y^{n_1}-x^{m_1}=0$.

Since $\text{cont}(C_f,C_h)>\frac{\lambda}{n}$, by Remark \ref{triangular} we get $$\text{cont}(C_f,C_h)>\text{cont}(C_f,C)=\text{cont}(C_h,C),$$ for any $C\in\mathcal{B}$. So,
$\text{max}_{C\in \mathcal B}	\{\text{cont}(C_f,C)\}=\text{max}_{C\in \mathcal B}	\{\text{cont}\big(C_h,C \big)\}$ and, consequently
$\frac{\lambda}{n}=\frac{\lambda'}{n'}$.

\item[ii)] Let us denote $\text{cont}(C_f,C_h)=\theta$. Since $x=0$ is transversal to $C_f$ and $C_h$, we get $\theta\geq 1$.  We will show that $\theta>\frac{\lambda}{n}$.

By hypothesis we get 
\begin{equation}\label{hypo}
	\textup{I}(C_f,C_h)>n'\cdot\left (\frac{(n_1-1)m+\lambda}{n_1}\right ).\end{equation}
Let us suppose by absurd that  $\theta\leq\frac{\lambda}{n}$. We have the following possibilites:
\begin{itemize}
	\item[a)] If $\frac{n}{n}\leq\theta<\frac{m}{n}$ then, by Proposition \ref{cont-I}, we get $\frac{\textup{I}(C_f,C_h)}{n'}=n\theta$. Since $m<\lambda$, we have that $$\hspace{2cm}\textup{I}(C_f,C_h)=n'n\theta<n'm=n'\frac{n_1m}{n_1}<n'\cdot\left ( \frac{(n_1-1)m+\lambda}{n_1}\right )$$ 
that contradicts (\ref{hypo}).

\item[b)] If $\frac{m}{n}\leq\theta\leq\frac{\lambda}{n}$ and $\lambda<\beta_2$, Proposition \ref{cont-I} give us 
$$\hspace{2cm}\textup{I}(C_f,C_h)=n'\cdot\left ( \frac{(n_1-1)m+n\theta}{n_1}\right )\leq n'\cdot\left ( \frac{(n_1-1)m+\lambda}{n_1}\right ),$$ 
that is an absurd since we have (\ref{hypo}).

\item[c)] Finally, if $\theta=\frac{\lambda}{n}$ and $\lambda=\beta_2$, by Proposition \ref{cont-I}, we get
$$\hspace{2cm}\textup{I}(C_f,C_h)=n'\cdot\left ( \frac{n_2v_2+n\theta-\beta_2}{n_1n_2}\right )= n'\cdot \frac{v_2}{n_1}=n'\cdot \left ( \frac{(n_1-1)m+\lambda}{n_1}\right ),$$ 
that it is not possible according to (\ref{hypo}).
\end{itemize}

So, $\text{cont}(C_f,C_h)=\theta>\frac{\lambda}{n}$ and the result follows by item $i)$.
\end{itemize}
\end{proof}	

\begin{ejemplo}\label{exemplo3}
Consider the plane branches $C_1\in K(3,7)$ defined by the Puiseux parametrization $(t^3,t^7+t^8)$ and $C_2\in K(6,14,17)$ given by
\begin{eqnarray*}
f_2 & = & y^6-6x^5y^4-2x^7(1+4x)y^3+9x^{10}(1-x)y^2+ \\ & + & 6x^{12}(1+x-x^2)y+x^{14}(1-x+10x^2-x^3).
\end{eqnarray*}

Once $8+3\not\in\langle 3,7\rangle$, the parametrization of $C_1$ is given as (\ref{zariski}) and, by definition, the Zariski invariant of $C_1$ is $\lambda_1=8$. 

Since 
\[\textup{I}(C_1,C_2)=\text{mult}(f_2(t^3,t^7+t^8))=45>44=6\cdot \frac{(3-1)\cdot 7+8}{3},\]
by item ii) of Proposition \ref{proposition-mais}, it follows that the Zariski invariant $\lambda_2$ of $C_2$ satisfies $\frac{\lambda_2}{6}=\frac{8}{3}$. So, we conclude that $\lambda_2=16$.
\end{ejemplo}

Up to a change of coordinates we can assume that $C_f$ is given by a Weierstrass polynomial $f=y^n+\sum_{i=1}^nc_i(x)y^{n-i}\in\mathbb{C}\{x\}[y]$ such that $$n=v_0=\textup{I}(f,x)=\min\{\textup{I}(C_f,C) \ :\ C\ \mbox{is a regular curve} \}$$ and $$m=v_1=\textup{I}(f,y)=\max\{\textup{I}(C_f,C) \ :\ C\ \mbox{is a regular curve} \}.$$

Let $C_h\in \mathcal{B}\subset K(n_1,m_1)$ be a plane branch analytically equivalent to $y^{n_1}-x^{m_1}=0$ and such that $\textup{I}(C_f,C_h)=(n_1-1)m+\lambda_f$, or equivalently $\text{cont}(C_f,C_h)=\frac{\lambda_f}{n}$. Since $C_f$ is given by a Weierstrass polynomial, we can consider $C_h$ defined by a monic polynomial $h\in\mathbb{C}\{x\}[y]$ with degree (and multiplicity) equal to $n_1=\frac{n}{e_1}$. In addition, systematically applying Euclidean division by $h$, that is, considering the $h$-expansion of $f$, we obtain $A_k\in\mathbb{C}\{x\}[y]$ with $deg_y(A_k)<n_1$ such that 
$$f=h^{e_1}+\sum_{k=0}^{e_1-1}A_kh^k\ \ \mbox{and}\ \ \textup{I}(f,h)=\textup{I}(A_0,h)=(n_1-1)m+\lambda_f.$$

Notice that the conductor $\mu_h$ of $\langle n_1,m_1\rangle$ is $\mu_h=(n_1-1)(m_1-1)<(n_1-1)m+\lambda_f$ then any integer $z\geq(n_1-1)m+\lambda_f$ belongs to $\langle n_1,m_1\rangle$. 
Thus, by Remark \ref{remark-std-expression}, there exist $p, q\in\mathbb{N}$ with $0\leq q<n_1$ such that
$$pn_1+qm_1=(n_1-1)m+\lambda_f=\textup{I}(A_0,h).$$

Since $\textup{I}(f,h)=(n_1-1)m+\lambda_f$, $\textup{I}(f,x)=n$ and $\textup{I}(f,y)=m$ it follows, by Remark \ref{triangular}, that $\textup{I}(h,x)=n_1$ and $\textup{I}(h,y)=m_1$. In this way, we get $\textup{I}(h,x^py^q)=(n_1-1)m+\lambda_f$ and consequently, there exist a unique $0\neq c\in\mathbb{C}$ and $h_1\in\mathbb{C}\{x\}[y]$ with $\text{deg}_yh_1<n_1$ such that $A_0=cx^py^q+h_1$ and $\textup{I}(h,h_1)>(n_1-1)m+\lambda_f=pn_1+qm_1$.

Similarly, we can write $h_1=\sum_{in_1+jm_1>pn_1+qm_1}a_{i,j}x^iy^j$ with $j<n_1$. We have proved the following result.

\begin{prop}\label{final-prop} Let $C_f\in K(n,m,\beta_2,\ldots ,\beta_g)$ be a plane branch defined by a Weierstrass polynomial $f\in\mathbb{C}\{x\}[y]$ with Zariski invariant $\lambda_f$, $\textup{I}(f,x)=n$ and $\textup{I}(f,y)=m$. There exist $C_h\in\mathcal{B}\subset K(n_1,m_1)$ with $\textup{I}(C_f,C_h)=(n_1-1)m+\lambda_f=pn_1+qm_1$ and a unique $0\neq c\in\mathbb{C}$ such that $$f=h^{e_1}+\sum_{k=1}^{e_1-1}A_kh^{k}+cx^py^q+\sum_{in_1+jm_1>pn_1+qm_1\atop j<n_1}a_{ij}x^iy^j,$$ where $A_k\in\mathbb{C}\{x\}[y]$ with $deg_y(A_k)<n_1$. \end{prop}

\begin{ejemplo}
Let us consider the plane branch $C_2\in K(6,14,17)$ given by
\begin{eqnarray*}
f_2 & = & y^6-6x^5y^4-2x^7(1+4x)y^3+9x^{10}(1-x)y^2+ \\ & + & 6x^{12}(1+x-x^2)y+x^{14}(1-x+10x^2-x^3).
\end{eqnarray*}

By Example \ref{exemplo3}, we have that the Zariski invariant of $C_2$ is $\lambda_2=16$.

Considering the plane branch  $C_h\in \mathcal{B}\subset K(3,7)$ given by $h=y^3-x^7$ with Puiseux parametrization $(t^3,t^7)$ we get
\[	f_2(t^3,t^7)= 9t^{44}-9t^{45}+6t^{46}-9t^{47}+10t^{48}-6t^{49}-t^{51}.
\]
In this way, we obtain
\[\textup{I}(C_2,C_h)=\text{mult}(f_2(t^3,t^7))=44=(3-1)\cdot 14+16.\]

According to Proposition \ref{final-prop}, we can express
$f_2=h^{e_1}+\sum_{k=1}^{e_1-1}A_kh^{k}+cx^py^q+h_1$ where $A_k\in\mathbb{C}\{x\}[y]$ with $deg_y(A_k)<n_1$, $pn_1+qm_1=44$, $c\neq 0$ and $h_1=\sum_{(i,j)}a_{i,j}x^iy^j$ where each $(i,j)$ satisfies $in_1+jm_1>44$ and $j<n_1$. In fact, since $n_1=3$, $m_1=7$ and $e_1=2$ we get $44=10\cdot 3+2\cdot 7$, that is, $p=10$ and $q=2$, that give us
\[f_2=h^2-(6x^5y+8x^8)h+9x^{10}y^2+h_1\] where $h_1=-9x^{11}y^2+(6x^{13}-6x^{14})y-9x^{15}+10x^{16}-x^{17}.
$

Notice that the plane branch $C_h\in\mathcal{B}\in K(n_1,m_1)$ in Proposition \ref{final-prop} is not unique. In fact, let us consider $C_{h'}$ given by $h'=y^3-3x^3y^2+3x^6y-x^7-x^9$ that admits Puiseux parametrization $\varphi(t)=(t^3,t^7+t^9)$. Since the change of coordinates $\sigma(x,y)=(x,y-x^2)$ is such that $\sigma\circ\varphi(t)=(t^3,t^7)$ it follows that $C_{h'}\in \mathcal{B}\subset K(3,7)$. Moreover, we get $\textup{I}(C_2,C_{h'})=\text{mult}(f_2(t^3,t^7+t^9))=44=(3-1)\cdot 14+16=(n_1-1)m+\lambda_2$. So, applying Proposition \ref{final-prop} we obtain
$$f_2=(h')^2+(6x^3y^2+3x^6y-6x^5-26x^8+11x^9)h'+9x^{10}y^2+h'_1$$ where $h'_1=(-69x^{11}+15x^{12})y^2+(15x^{13}+66x^{14}-24x^{15})y-27x^{15}+19x^{16}-27x^{17}+10x^{18}$.
\end{ejemplo}

In particular, if $C_f\in K(n,m)$ and, considering a change of coordinates, such that the Puiseux parametrization of $C_f$ is given by (\ref{zariski}), then $e_1=1$, $h=y^n-x^m$ and Proposition \ref{final-prop} give us
$$f=y^n-x^m+cx^py^q+\sum_{in+jm>pn+qm\atop j<n}a_{ij}x^iy^j$$
that is a similar expression to the one considered by Peraire in \cite{Peraire}.

\vspace{0.75cm}

\noindent {\sc Marcelo Escudeiro Hernandes}\\
	Universidade Estadual de Maring\'a. \\
	Av. Colombo 5790. Maring\'a-PR 87020-900,
	Brazil. \\ORCID ID: 0000-0003-0441-8503
	
\noindent {mehernandes@uem.br}
	\vspace{0.3cm}

\noindent {\sc Mauro Fernando Hern\'andez Iglesias}\\
	Pontificia Universidad Cat\'olica del Per\'u. \\
	Av. Universitaria 1801, San Miguel 15088, Per\'u. \\ ORCID ID: 0000-0003-0026-157X
	
\noindent {mhernandezi@pucp.pe}

\end{document}